\documentclass[11pt]{article}

\usepackage[utf8]{inputenc}
\usepackage[english]{babel}
\usepackage{amsmath}
\usepackage{nicefrac}
\usepackage{amsthm}
\usepackage{amsfonts}
\usepackage{bbm}

\newcommand{\R}{\mathbb{R}}

\newcommand{\N}{\mathbb{N}}

\newcommand{\F}{\mathcal{F}}

\renewcommand{\P}{\mathbb{P}}
\newcommand{\E}{\mathbb{E}}

\newcommand{\1}{\mathbbm{1}}

\newtheorem{Theorem}{Theorem}[section]
\newtheorem{Proposition}[Theorem]{Proposition}
\newtheorem{Corollary}[Theorem]{Corollary}

\newtheorem{Remark}[Theorem]{Remark}
\newtheorem{Definition}[Theorem]{Definition}
\newtheorem{Example}[Theorem]{Example}

\numberwithin{equation}{section}

\newenvironment{sciabstract}{\begin{quote}}{\end{quote}}

\newcounter{lastnote}

\title{Boundary behavior of multi-type continuous-state branching processes with immigration}
\newcommand{\pdftitle}{Boundary behavior of multi-type continuous-state branching processes with immigration}
\newcommand{\pdfauthor}{Martin Friesen}

\author{
Martin Friesen\footnote{Fakult\"at f\"ur Mathematik und Naturwissenschaften, Bergische Universit\"at Wuppertal, Gaußstraße 20, 42119 Wuppertal, Germany, friesen@math.uni-wuppertal.de}\\
Peng Jin\footnote{Department of Mathematics, Shantou University, Shantou, Guangdong 515063, China, pjin@stu.edu.cn \newline
Peng Jin is supported by the STU Scientific Research Foundation for Talents (No. NTF18023)}\\
Barbara R\"udiger\footnote{Fakult\"at f\"ur Mathematik und Naturwissenschaften, Bergische Universit\"at Wuppertal, Gaußstraße 20, 42119 Wuppertal, Germany, ruediger@uni-wuppertal.de}
}

\usepackage[plainpages=false,pdfpagelabels=true,bookmarks=true,pdfauthor={\pdfauthor},
pdftitle={\pdftitle}]{hyperref}

\makeatletter
\def\HyPsd@CatcodeWarning#1{}
\makeatother

\begin{document}

\maketitle

\begin{sciabstract}\textbf{Abstract:}
In this article we provide
a sufficient condition for a continuous-state branching
process with immigration (CBI process) to not hit its boundary, 
i.e. for non-extinction.
Our result applies to arbitrary dimension $d\geq1$ and is formulated
in terms of an integrability condition for its immigration and branching
mechanisms $F$ and $R$. The proof is based on a suitable comparison
with one-dimensional CBI processes and an existing result for one-dimensional
CBI processes. The same technique is also used to provide a sufficient
condition for transience of multi-type CBI processes.
\end{sciabstract}


\section{Introduction}
Continuous-state branching processes with immigration
(shorted as CBI processes) form a class of time-homogeneous Markov
processes with state space
\[
\R_{+}^{d}=\{x\in\R^{d}\ |\ x_{1},\dots,x_{d}\geq0\},\ \ d\in\N,
\]
whose Laplace transform is an exponentially affine function of the
initial state variable, i.e., CBI processes are affine
processes in the sense of \cite[Definition 2.6]{DFS03}.
They have been first studied in dimension $d=1$ in \cite{F51}, \cite{L67b} and \cite{SW73}, where it was shown
that they arise as scaling limits of Galton-Watson branching processes.
For an introduction to such type of processes
in arbitrary dimension we refer to \cite{K06}, \cite{P16} and
\cite{L11}, where superprocesses were also discussed. Although these
processes are initially used to describe populations
of multiple spices, they have also various applications in mathematical
finance, see, e.g., \cite{A15} and \cite{DFS03} and
the references therein. At this point we would like to mention only
some recent results on the long-time behavior of CBI processes.
Namely, convergence of supercritical CBI processes was recently studied in 
\cite{BPP18b} and \cite{BPP18} while convergence in the total variation distance for affine processes on convex cones (including subcritical CBI processes)
was recently studied in \cite{MSV18}.
Results applicable to the class of affine processes on the canonical state space $\R_+^d \times \R^n$ were obtained in \cite{FJR18c}, \cite{GZ18} and \cite{JKR18}.

Let us describe CBI processes in more detail. 
\begin{Definition}\label{ADMISSIBLE}
The tuple $(c,\beta,B,\nu,\mu)$ is called admissible if
\begin{enumerate}
\item[(i)] $c=(c_{1},\dots,c_{d})\in\R_{+}^{d}$.
\item[(ii)] $\beta=(\beta_{1},\dots,\beta_{d})\in\R_{+}^{d}$.
\item[(iii)] $B=(b_{kj})_{k,j\in\{1,\dots,d\}}$ is such that,
for $k,j\in\{1,\dots,d\}$ with $k \neq j$, one has
\[
b_{kj}-\int\limits _{\R_{+}^{d}}z_{k}\mu_{j}(dz) \geq 0.
\]
\item[(iv)] $\nu$ is a Borel measure on $\R_{+}^{d}$ satisfying $\int_{\R_{+}^{d}}(1\wedge|z|)\nu(dz)<\infty$
and $\nu(\{0\})=0$.
\item[(vi)] $\mu=(\mu_{1},\dots,\mu_{d})$, where, for each $j\in\{1,\dots,d\}$,
$\mu_{j}$ is a Borel measure on $\R_{+}^{d}$ satisfying
\begin{equation}
\int\limits _{\R_{+}^{d}}\left(|z|\wedge|z|^{2}+\sum\limits _{k\in\{1,\dots,d\}\backslash\{j\}}z_{k}\right)\mu_{j}(dz)<\infty,\ \ \mu_{j}(\{0\})=0.\label{EQ:00}
\end{equation}
\end{enumerate}
\end{Definition} 
Note that this definition is a special
case of \cite[Definition 2.6]{DFS03}. Here we consider the state
space $\R_{+}^{d}$, exclude killing and require the measures $\mu_{1},\dots,\mu_{d}$
to satisfy the additional integrability condition $\sum_{j=1}^{d}\int_{|z|>1}|z|\mu_{j}(dz)<\infty$,
see also \cite[Remark 2.3]{BLP15} for additional comments. These
conditions together imply that the multi-type CBI process introduced
below is conservative.

Let $(c,\beta,B,\nu,\mu)$ be admissible parameters. 
It was shown in \cite[Theorem 2.7]{DFS03} (see also \cite[Remark 2.5]{BLP15}),
that there exists a unique conservative Feller transition semigroup $(P_t)_{t \geq 0}$ acting on the Banach space of continuous functions vanishing at infinity with state space $\R_{+}^d$ such that its generator has core $C_c^{\infty}(\R_+^d)$ and is, for $f \in C_c^2(\R_+^d)$, given by 
\begin{align}\label{GENERATOR}
(Lf)(x) & =\sum\limits _{j=1}^{d}c_{j}x_{j}\frac{\partial^{2}f(x)}{\partial x_{j}^{2}}+\langle\beta+Bx,(\nabla f)(x)\rangle+\int\limits _{\R_{+}^{d}}(f(x+z)-f(x))\nu(dz)\\
 & \ \ \ +\sum\limits _{j=1}^{d}x_{j}\int\limits _{\R_{+}^{d}}\left(f(x+z)-f(x)-\langle z,(\nabla f)(x)\rangle\right)\mu_{j}(dz),
\end{align}
where $\langle\cdot,\cdot\rangle$ denotes the Euclidean scalar product
on $\R^{d}$. The corresponding Markov process with generator $L$ is called multi-type CBI process. Moreover, the Laplace transform of its transition kernel 
$P_t(x,dy)$ has representation
\[
\int\limits _{\R_{+}^{d}}e^{-\langle\xi,y\rangle}P_{t}(x,dy)=\exp\left(-\langle x,v(t,\xi)\rangle-\int\limits _{0}^{t}F(v(s,\xi))ds\right),\ \ x,\xi\in\R_{+}^{d},\ \ t\geq0,
\]
where, for any $\xi\in\R_{+}^{d}$, the continuously differentiable
function $t\longmapsto v(t,\xi)\in\R_{+}^{d}$ is the unique locally
bounded solution to the system of differential equations
\begin{equation}
\frac{\partial v(t,\xi)}{\partial t}=-R(v(t,\xi)),\ \ v(0,\xi)=\xi.\label{RICCATI}
\end{equation}
Here $F$ and $R$ are of L\'evy-Khinchine form
\begin{align*}
F(\xi) &= \langle\beta,\xi\rangle+\int\limits _{\R_{+}^{d}}\left(1-e^{-\langle\xi,z\rangle}\right)\nu(dz),\\
R_{j}(\xi) &= c_{j}\xi_{j}^{2}-\langle Be_{j},\xi\rangle+\int\limits _{\R_{+}^{d}}\left(e^{-\langle\xi,z\rangle}-1+\langle\xi,z\rangle\right)\mu_{j}(dz),\qquad j\in\{1,\dots,d\},
\end{align*}
and $e_{1},\dots,e_{d}$ denote the canonical basis vectors in $\R^{d}$.
Most of the results obtained for multi-type CBI processes are based
on a detailed study of the generalized Riccati equation \eqref{RICCATI},
where $F$ and $R$ are called the immigration and branching mechanisms, respectively.

The possibility to describe a multi-type CBI process as a strong solution
to a stochastic differential equation was studied in \cite{BLP15}.
Below we provide such a pathwise description.
Let $(\Omega,\F,(\F_{t})_{t\geq0},\P)$ be a filtered
probability space satisfying the usual conditions. Consider the following
objects defined on $(\Omega,\F,(\F_{t})_{t\geq0},\P)$:
\begin{enumerate}
\item[(A1)] A $d$-dimensional $(\F_{t})_{t\geq0}$-Brownian motion $W=(W(t))_{t\geq0}$.
\item[(A2)] $(\F_{t})_{t\geq0}$-Poisson random measures $N_{1},\dots,N_{d}$
on $\R_{+}\times\R_{+}^{d}\times\R_{+}$ with compensators
\[
\widehat{N}_{j}(ds,dz,dr)=ds\mu_{j}(dz)dr,\qquad j\in\{1,\dots,d\}.
\]
\item[(A3)] A $(\F_{t})_{t\geq0}$-Poisson random measure $N_{\nu}$ on $\R_{+}\times\R_{+}^{d}$
with compensator $\widehat{N}_{\nu}(ds,dz)=ds\nu(dz)$.
\end{enumerate}
The objects $W,N_{\nu},N_{1},\dots,N_{d}$ are supposed to be mutually
independent. Denote by $\widetilde{N}_{j}=N_{j}-\widehat{N}_{j}$,
$j\in\{1,\dots,d\}$, and $\widetilde{N}_{\nu}=N_{\nu}-\widehat{N}_{\nu}$
the corresponding compensated Poisson random measures. Then it was
shown in \cite[Theorem 4.6]{BLP15} that, for each $x\in\R_{+}^{d}$
there exists a unique $\R_{+}^{d}$-valued strong solution to
\begin{align}
X(t) & =x+\int\limits _{0}^{t}\left(\beta+BX(s)\right)ds+\sum\limits _{k=1}^{d}\sqrt{2c_{k}}e_{k}\int\limits _{0}^{t}\sqrt{X_{k}(s)}dW_{k}(s)+\int\limits _{0}^{t}\int\limits _{\R_{+}^{d}}zN_{\nu}(ds,dz)\label{SDE:CBI}\\
 & \nonumber \ \ \ +\sum\limits _{j=1}^{d}\int\limits _{0}^{t}\int\limits _{|z|\leq1}\int\limits _{\R_{+}}z\1_{\{r\leq X_{j}(s-)\}}\widetilde{N}_{j}(ds,dz,dr)\\
 & \ \ \ +\sum\limits _{j=1}^{d}\int\limits _{0}^{t}\int\limits _{|z|>1}\int\limits _{\R_{+}}z\1_{\{r\leq X_{j}(s-)\}}N_{j}(ds,dz,dr)
 -\sum\limits _{j=1}^{d}\int\limits _{0}^{t}\left(\int\limits _{|z|>1}z\mu_{j}(dz)\right)X_{j}(s)ds.\nonumber
\end{align} 
An application of the It\^{o}-formula shows that
$X$ solves the martingale problem with generator \eqref{GENERATOR},
i.e., $X$ is a multi-type CBI process. Conversely, the law of a multi-type
CBI process can be obtained from \eqref{SDE:CBI}, see \cite{BLP15}
for additional details.

Smoothness of transition probabilities for one-dimensional CBI processes
was recently studied in \cite{CLP18}, where very precise results
have been obtained. In \cite{FJR18a} (see also \cite{FMS13} for
related results) we have studied existence of transition densities
for multi-type CBI processes. It was shown that, under appropriate
conditions, such a density exists on the interior of its state space,
i.e. on $\Gamma=\{x\in\R_{+}^{d}\ |\ x_{1},\dots,x_{d}>0\}$.
In this work we provide conditions under which the corresponding multi-type
CBI process is supported on $\Gamma$, i.e. $\P[ X(t) \in \Gamma, \quad t \geq 0] = 1$. Such property simply states that the population described by $X$ does not get extinct. As a consequence, it has, under
the conditions of \cite{FJR18a} and those presented in this work,
a density on the whole state space $\R_{+}^{d}$.

The study of boundary behavior, recurrence and transience for CBI
processes has, in dimension $d=1$, a long history where we would
like to mention the works \cite{G74} and \cite{FFS85}. More recent
works, still in dimension $d=1$, include \cite{CPU13}, \cite{DFM14}, \cite{FU13},
 and \cite{FU14}. Based on these results we provide
sufficient conditions for non-extinction
and transience of multi-type CBI processes applicable in arbitrary
dimension $d\geq1$.

This work is organized as follows. In Section 2 we state and discuss
the main results of this work. These results are then
proved in Section 3, while some technical computations are given in
the appendix.

\section{Statement of the results}
Here and below we denote by $X$ a multi-type CBI process with admissible parameters $(c,\beta,B,\nu,\mu)$ obtained from \eqref{SDE:CBI}. 
We start with the simple case where one component
of the multi-type CBI process has bounded variation. 
\begin{Proposition}\label{PROP:00}
Suppose that there exists $k\in\{1,\dots,d\}$ such that
\begin{equation}
c_{k}=0\qquad\text{and}\qquad\int\limits _{|z|\leq1}z_{k}\mu_{k}(dz)<\infty.\label{EQ:07}
\end{equation}
Then $X_{k}$ has bounded variation and 
\begin{equation}
X_{k}(t)\geq\begin{cases}
e^{\theta_{k}t}x_{k}+\beta_{k}\frac{e^{\theta_{k}t}-1}{\theta_{k}}, & \text{ if }\theta_{k}\neq0\\
x_{k}+\beta_{k}t, & \text{ if }\theta_{k}=0
\end{cases},\qquad t\geq0,\label{EQ:09}
\end{equation}
where $\theta_{k}= b_{kk} - \int _{\R_+^d}z_{k}\mu_{k}(dz)\in\R$.
\end{Proposition} 
The proof of this result is given in the appendix.
From this we easily obtain the following corollary. 
\begin{Corollary}
Let $k \in \{1,\dots, d\}$ and suppose that \eqref{EQ:07} holds.
If either $x_k > 0$ or $\beta_k > 0$,
then $\P[X_{k}(t) > 0, \ \ t \geq 0] = 1$.
\end{Corollary} 
The next proposition gives a multi-dimensional analogue
of this result. For $x,y\in\mathbb{R}^{d}$ we will
write $x\le y$ to mean that $x_{i}\le y_{i}$ for all $i=1,\ldots,d$.
\begin{Proposition}\label{PROP:01} 
Suppose that \eqref{EQ:07} holds
for all $k\in\{1,\dots,d\}$. Then $X$ has bounded variation and
it holds that
\begin{equation}
X(t)\geq e^{tG}x+\int\limits _{0}^{t}e^{sG}\beta ds,\label{EQ:10}
\end{equation}
where $G=(g_{kj})_{k,j\in\{1,\dots,d\}}$ is given by
\begin{align}\label{GDEF}
 g_{kj}=b_{kj} - \int \limits_{\R_+^d}z_k \mu_j(dz), \qquad k,j \in \{1,\dots, d\}.
\end{align}
\end{Proposition} The proof of this statement is given in the appendix.
In view of this estimate we restrict our further analysis to the
case where \eqref{EQ:07} does not hold, i.e., the process has unbounded
variation. In this case we define,
for $k \in \{1,\dots, d\}$, the projected immigration and branching mechanisms $F^{(k)},R^{(k)}:\R\longrightarrow\R$ by
\begin{align*}
F^{(k)}(\xi) & =\beta_{k}\xi+\int\limits _{\R_{+}^{d}}\left(1-e^{-\xi z_{k}}\right)\nu(dz),\\
R^{(k)}(\xi) & =-b_{kk}\xi+c_{k}\xi^{2}+\int\limits _{\R_{+}^{d}}\left(e^{-\xi z_{k}}-1+\xi z_{k}\right)\mu_{k}(dz).
\end{align*}
Then we obtain the following result. 
\begin{Theorem}\label{MAIN:THEOREM}
Suppose that there exists $k\in\{1,\dots,d\}$ and 
$\kappa>0$ such that $R^{(k)}(\xi)>0$ for $\xi \geq \kappa$.
If $c_{k}>0$ or $\int_{|z|\leq1}z_{k}\mu_{k}(dz)=\infty$,
and it holds that
\begin{equation}
\int\limits _{\kappa}^{\infty}\exp\left(\int\limits _{\kappa}^{\xi}\frac{F^{(k)}(u)}{R^{(k)}(u)}du\right)\frac{1}{R^{(k)}(\xi)}d\xi=\infty,\label{EQ:04}
\end{equation}
then $\P[ X_k(t) > 0, \ \ t \geq 0] = 1$, provided $x_k > 0$.
\end{Theorem}
From this we directly deduce the following corollary.
\begin{Corollary}
If for each $k\in\{1,\dots,d\}$ the conditions of Theorem \ref{MAIN:THEOREM}
are satisfied, then $\P[ X(t) \in \Gamma, \quad t \geq 0] = 1$,
provided $x \in \Gamma = \{ x \in \R_+^d \ | \ x_1,\dots, x_d > 0 \}.$
\end{Corollary} 
We close
this subsection with a sufficient condition for \eqref{EQ:04}.
\begin{Remark}\label{CORR:01} 
Suppose that for some $k\in\{1,\dots,d\}$ the following conditions are satisfied:
\begin{enumerate}
\item[(i)] There exists $M_0>0$ such that $R^{(k)}(\xi)>0$ for $\xi\geq M_0$.
\item[(ii)] There exists $\gamma_{k}\in(0,1]$ and $M_1,C_{1}>0$ such that
 $F^{(k)}(\xi)\geq C_{1}\xi^{\gamma_{k}}$ for $\xi \geq M_1$.
\item[(iii)] There exist $\alpha_{k}\in(1,2]$ and $M_2,C_{2}>0$ such that
 $R^{(k)}(\xi)\leq C_{2}\xi^{\alpha_{k}}$ for $\xi \geq M_2$.
\end{enumerate}
Then \eqref{EQ:04} is satisfied, provided one of the following conditions
holds:
\begin{enumerate}
\item[(a)] $\alpha_{k}\in(1,1+\gamma_{k})$.
\item[(b)] $\alpha_{k}=1+\gamma_{k}$ and $\gamma_{k}\leq\frac{C_{1}}{C_{2}}$.
\end{enumerate}
\end{Remark} 
The proof of this remark is given in the appendix. Note
that, if $\beta_{k}>0$, then $F^{(k)}(\xi)\geq\beta_{k}\xi$ and
hence $\gamma_{k}=1$. However, this corollary also applies in the
particular case where $\beta_{1}=\dots=\beta_{d}=0$.

Finally we close our considerations with one sufficient condition for transience.
\begin{Theorem}\label{MAIN:THEOREM1} 
Let $k \in \{1,\dots, k\}$ and suppose that 
$R^{(k)}(\xi)>0$ holds for all $\xi>0$. 
Then $\P[\lim_{t\to\infty}X_{k}(t) = \infty]=1$, provided one of the following conditions is satisfied:
\begin{enumerate}
\item[(a)] $b_{kk}>0$.
\item[(b)] $b_{kk}\leq0$ and 
\begin{equation}
\int\limits _{0}^{1}\exp\left(-\int\limits _{\xi}^{1}\frac{F^{(k)}(u)}{R^{(k)}(u)}du\right)\frac{d\xi}{R^{(k)}(\xi)}<\infty.\label{EQ:16}
\end{equation}
\end{enumerate}
\end{Theorem} 
From this we easily conclude that, if the assumptions
of Theorem \ref{MAIN:THEOREM1} hold for each $k\in\{1,\dots,d\}$,
then $X$ is transient.

Let us close this section with one particlar example.
The multi-type CBI process $X$ with admissible
parameters $(c=0,\beta,B,\nu,\mu)$, where $\mu=(\mu_{1},\dots,\mu_{d})$
are, for $\alpha_{1},\dots,\alpha_{d}\in(1,2)$, given by
\begin{equation}
\mu_{j}(dz)=\1_{\R_{+}}(z_{j})\frac{dz_{j}}{z_{j}^{1+\alpha_{j}}}\otimes\prod\limits _{k\neq j}\delta_{0}(dz_{k}),\label{EQ:06}
\end{equation}
is called $d$-dimensional anisotropic $(\alpha_{1},\dots,\alpha_{d})$-root
process. 
\begin{Theorem}\label{THEOREM:02}
Let $X$ be the anisotropic $(\alpha_{1},\dots,\alpha_{d})$-root
process starting from $x \in \R_+^d$. Fix $k\in\{1,\dots,d\}$.
\begin{enumerate}
 \item[(a)] Suppose that there exist $C,M>0$ and $\gamma_{k}\in(0,1]$ such
that
\begin{equation}
\beta_{k}\xi+\int\limits _{\R_{+}^{d}}\left(1-e^{-\xi z_{k}}\right)\nu(dz)\geq C\xi^{\gamma_{k}},\qquad\xi\geq M.\label{EQ:13}
\end{equation}
If $x_{k}>0$ and $\alpha_{k}\in(1,1+\gamma_{k})$, 
then $\P[ X_k(t) > 0, \quad t \geq 0] = 1$.
 \item[(b)] If $b_{kk} > 0$, then $\P[\lim_{t\to\infty}X_{k}(t)=\infty]=1$.
\end{enumerate}
\end{Theorem} 
\begin{proof} 
 Assertion (b) follows immediately from Theorem
\ref{MAIN:THEOREM1} (a). Let us prove assertion (a).
 Since $\alpha_{1},\dots,\alpha_{d}\in(1,2)$, it follows that $X$
has unbounded variation. Hence it suffices to show that
Theorem \ref{MAIN:THEOREM} is applicable. First observe that
\begin{align*}
F^{(k)}(\xi) & =\beta_{k}\xi+\int\limits _{\R_{+}^{d}}\left(1-e^{-\xi z_{k}}\right)\nu(dz),\\
R^{(k)}(\xi) & =-b_{kk}\xi+\int\limits _{0}^{\infty}\left(e^{-\xi z}-1+\xi z\right)\frac{dz}{z^{1+\alpha_{k}}}=-b_{kk}\xi+K\xi^{\alpha_{k}},
\end{align*}
where $K=\int_{0}^{\infty}\left(e^{-w}-1+w\right)\frac{dw}{w^{1+\alpha_{k}}}>0$.
Next it is easily seen that
\[
R^{(k)}(\xi)>0,\ \ \text{ whenever }\xi>\left(\frac{\max\{0,b_{kk}\}}{K}\right)^{\frac{1}{\alpha_{k}-1}}.
\]
Moreover, one finds $R^{(k)}(\xi)\leq\left(|b_{kk}|+K\right)\xi^{\alpha_{k}}$
for $\xi\geq1$, and hence the assertion follows from Remark \ref{CORR:01}
since $\alpha_{k}\in(1,1+\gamma_{k})$. 
\end{proof} 
In Remark \ref{CORR:01}, if $\beta_{k}>0$, then we may take $\gamma_{k}=1$ 
so that \eqref{EQ:13} is satisfied. 
However, if $\beta_{k}=0$, then \eqref{EQ:13} may be still satisfied
as it is shown in the following example. 
\begin{Example} Let $\gamma\in(0,1)$
and set $\nu(dz)=\1_{\R_{+}^{d}}(z)\frac{dz}{|z|^{d+\gamma}}$. Then
$\int_{\R_{+}^{d}}(1\wedge|z|)\nu(dz)<\infty$ and
\[
\int\limits _{\R_{+}^{d}}\left(1-e^{-\xi z_{k}}\right)\frac{dz}{|z|^{d+\gamma}}=\xi^{\gamma}\int\limits _{\R_{+}^{d}}\left(1-e^{-w_{k}}\right)\frac{dw}{|w|^{d+\gamma}}.
\]
So \eqref{EQ:13} holds for $\gamma_{k}=\gamma$. 
Hence the assumptions of Theorem \ref{THEOREM:02} (a)
are satisfied, if $\alpha_{k}\in(1,1+\gamma)$. 
\end{Example} It
is worthwhile to mention that there exists a large class of measures
which satisfy \eqref{EQ:13} but are not of the form $\nu(dz)=\1_{\R_{+}^{d}}(z)\frac{dz}{|z|^{d+\gamma}}$,
see, e.g., \cite{KS17}, \cite{FJR18a} and \cite{FJR18b}. 

\section{Proofs of main results}

\subsection{Construction of auxilliary CBI process}

Let $(c,\beta,B,\nu,\mu)$ be admissible parameters and set
\begin{equation}
\widetilde{b}_{kj} = b_{kj}-\int\limits _{|z|>1}z_{k}\mu_{j}(dz)-\1_{\{k\neq j\}}\int\limits _{|z| \leq 1}z_{k}\mu_{j}(dz).\label{EQ:15}
\end{equation}
Let $(W,N_{\nu},N_{1},\dots,N_{d})$ be given as in (A1) -- (A3)
and consider a process $Y=(Y_{1},\ldots,Y_{d})$
satisfying, for each $k=1,\ldots,d$, the stochastic equation
\begin{align}
Y_{k}(t) & =y_{k}+\int\limits _{0}^{t}\left(\beta_{k}+\widetilde{b}_{kk}Y_{k}(s)\right)ds+\sqrt{2c_{k}}\int\limits _{0}^{t}\sqrt{Y_{k}(s)}dW_{k}(s)+\int\limits _{0}^{t}\int\limits _{\R_{+}^{d}}z_{k}N_{\nu}(ds,dz)\label{EQ:12}\\
 & \ \ \ +\int\limits _{0}^{t}\int\limits _{|z|\leq1}\int\limits _{\R_{+}}z_{k}\1_{\{r\leq Y_{k}(s-)\}}\widetilde{N}_{k}(ds,dz,dr)+\int\limits _{0}^{t}\int\limits _{|z|>1}\int\limits _{\R_{+}}z_{k}\1_{\{r\leq Y_{k}(s-)\}}N_{k}(ds,dz,dr),\nonumber
\end{align}
where $y=(y_{1},\ldots,y_{d})\in\R_{+}^{d}$. Finally, define projection
mappings $\mathrm{pr}_{j}:\R_{+}^{d}\longrightarrow\R_{+}$, $\mathrm{pr}_{j}(z)=z_{j}$,
$j\in\{1,\dots,d\}$. The next lemma states that the
system of equations \eqref{EQ:12} has a unique strong solution which
describes a CBI process. 
\begin{Proposition}\label{LEMMA:00} Let
$(c,\beta,B,\nu,\mu)$ be admissible parameters and let $(W,N_{\nu},N_{1},\dots,N_{d})$
be given as in (A1) -- (A3). Then the following hold:
\begin{enumerate}
\item[(a)] For each $y\in\R_{+}^{d}$, there exists a unique $\R_{+}^{d}$-valued
strong solution $Y$ to \eqref{EQ:12}.
\item[(b)] For each $j\in\{1,\dots,d\}$, $Y_{j}$ is a one-dimensional
CBI process with admissible parameters $(c_{j},\beta_{j},b_{jj},\widetilde{\nu}_{j},\widetilde{\mu}_{j})$,
where $\widetilde{\nu}_{j}=\nu\circ\mathrm{pr}_{j}^{-1}$, $\widetilde{\mu}_{j}=\mu_{j}\circ\mathrm{pr}_{j}^{-1}.$
\end{enumerate}
\end{Proposition} \begin{proof} 
Define random measures $M_{1}(ds,dz,dr),\dots,M_{d}(ds,dz,dr)$
on $\R_{+}^{3}$ by
\[
M_{k}((a,b]\times A\times B)=N_{k}((a,b]\times\mathrm{pr}_{k}^{-1}(A)\times B),\qquad k\in\{1,\dots,d\},
\]
and $N_{\widetilde{\nu}_{1}}(ds,dz),\dots,N_{\widetilde{\nu}_{d}}(ds,dz)$
on $\R_{+}^{2}$ by
\[
N_{\widetilde{\nu}_{k}}((a,b] \times A)=N_{\nu}((a,b]\times\mathrm{pr}_{k}^{-1}(A)),\qquad k\in\{1,\dots,d\},
\]
where $a<b$, $A,B\in\mathcal{B}(\R_{+})$. Then $M_{1},\dots,M_{d}$
and $N_{\widetilde{\nu}_{1}},\dots,N_{\widetilde{\nu}_{d}}$ are Poisson
random measures with compensators
\[
\widehat{M}_{k}(ds,dz,dr)=ds\widetilde{\mu}_{k}(dz)dr\ \text{ and }\ \widehat{N}_{\widetilde{\nu}_{k}}(ds,dz)=ds\widetilde{\nu}_{k}(dz),\quad k\in\{1,\dots,d\}.
\]
Moreover, $M_{k},N_{\widetilde{\nu}_k},W_k$ are mutually independent.
Let $\widetilde{M}_{k}(ds,dz,dr)=M_{k}(ds,dz,dr)-\widehat{M}_{k}(ds,dz,dr)$
be the corresponding compensated Poisson random measures. Then \eqref{EQ:12}
takes the form
\begin{align*}
Y_{k}(t) & =y_{k}+\int\limits _{0}^{t}\left(\beta_{k}+\widetilde{b}_{kk}Y_{k}(s)\right)ds+\sqrt{2c_{k}}\int\limits _{0}^{t}\sqrt{Y_{k}(s)}dW_{k}(s)+\int\limits _{0}^{t}\int\limits _{\R_{+}}zN_{\widetilde{\nu}_{k}}(ds,dz)\\
 & \ \ \ +\int\limits _{0}^{t}\int\limits _{(0,1]}\int\limits _{\R_{+}}z\1_{\{r\leq Y_{k}(s-)\}}\widetilde{M}_{k}(ds,dz,dr)+\int\limits _{0}^{t}\int\limits _{(1,\infty)}\int\limits _{\R_{+}}z\1_{\{r\leq Y_{k}(s-)\}}M_{k}(ds,dz,dr).
\end{align*}
This equation is now a particular case of \eqref{SDE:CBI} for dimension
$d=1$, i.e., it has a unique $\R_{+}$-valued solution which is a
CBI process with admissible parameters $(c_{k},\beta_{k},\widetilde{b}_{kk},\widetilde{\nu}_{k},\widetilde{\mu}_{k})$,
see also \cite{MR2584896} for related results. 
\end{proof}
We close this section with the observation that $Y$ obtained from \eqref{EQ:12}
is actually a CBI process on $\R_+^d$.
\begin{Remark}
 Let $(c,\beta,B,\nu,\mu)$ be admissible parameters, 
 let $(W,N_{\nu},N_{1},\dots,N_{d})$ be given as in (A1) -- (A3),
 and let $Y$ be the unique solution to \eqref{EQ:12}.
 Then $Y$ is a multi-type CBI process
with admissible parameters $(c,\beta,B^{Y},\nu,\mu^{Y})$, where $B^{Y}=\mathrm{diag}(b_{11},\dots,b_{dd})$
and $\mu^{Y}=(\mu_{1}^{Y},\dots,\mu_{d}^{Y})$ with $\mu_{j}^{Y}(dz)=\widetilde{\mu}_{j}(dz_{k})\otimes\prod_{k\neq j}\delta_{0}(dz_{k})$,
$j=1,\ldots,d$.
\end{Remark}
 Since we do not use this result later on, we only sketch the main idea of proof. In view of \cite{BLP15} it suffices to show that the Markov generator of $Y$ takes the desired form. However, this can be shown by direct computation
 using It\^{o}'s formula.

\subsection{Comparison with auxiliary CBI process}

The next statement is the key estimate for this work. 
\begin{Proposition}\label{COMPARISON}
Let $(c,\beta,B,\nu,\mu)$ be admissible parameters. Consider $(W,N_{\nu},N_{1},\dots,N_{d})$
as in (A1) -- (A3), and let $X$ be the multi-type CBI process obtained
from \eqref{SDE:CBI}. Let
$Y$ be the unique strong solution to \eqref{EQ:12} with $y=x$.
Then
\[
\P[X_{k}(t)\geq Y_{k}(t),\ \ t\geq0]=1,\qquad k\in\{1,\dots,d\}.
\]
\end{Proposition} \begin{proof} Our proof is based on the method
developed in \cite[Lemma 4.1]{BLP15}. 
Define $\Delta_{k}(t):=Y_{k}(t)-X_{k}(t)$ and
$\delta_k(r,s-) = \1_{ \{r \leq Y_k(s-)\} } - \1_{ \{ r \leq X_k(s-)\} }$.
Then $\Delta_{k}(0)=0$ and we obtain, for each $k \in \{1,\dots, d\}$, 
\begin{align*}
\Delta_{k}(t) & =\int\limits _{0}^{t}\left(\widetilde{b}_{kk}\Delta_{k}(s)-\sum\limits _{j\neq k}\widetilde{b}_{kj}X_{j}(s)\right)ds+\sqrt{2c_{k}}\int\limits _{0}^{t}\left(\sqrt{Y_{k}(s)}-\sqrt{X_{k}(s)}\right)dW_{k}(s)\\
 & \ \ \ +\int\limits _{0}^{t}\int\limits _{|z|\leq1}\int\limits _{\R_{+}}z_{k}\delta_{k}(r,s-)\widetilde{N}_{k}(ds,dz,dr)+\int\limits _{0}^{t}\int\limits _{|z|>1}\int\limits _{\R_{+}}z_{k}\delta_{k}(r,s-)N_{k}(ds,dz,dr)\\
 & \ \ \ -\sum\limits _{j\neq k}\int\limits _{0}^{t}\int\limits _{\R_{+}^{d}}\int\limits _{\R_{+}}z_{k}\1_{\{r\leq X_{j}(s-)\}}N_{j}(ds,dz,dr).
\end{align*}
Let $\phi_{m}:\R\longrightarrow\R_{+}$ be a sequence of twice continuously
differentiable functions with the properties:
\begin{enumerate}
\item[(i)] $\phi_{m}(z)\nearrow z_{+}:=\max\{0,z\}$, as $m\to\infty$ for all
$z\in\R$.
\item[(ii)] $\phi_{m}'(z)\in[0,1]$ for all $m\in\N$ and $z\geq0$.
\item[(iii)] $\phi_{m}'(z)=\phi_{m}(z)=0$ for all $m\in\N$ and $z\leq0$.
\item[(vi)] $\phi_{m}''(x-y)(\sqrt{x}-\sqrt{y})^{2}\leq2/m$ for all $m\in\N$
and $x,y\geq0$.
\end{enumerate}
The existence of such a sequence was shown in the proof of \cite[Theorem 3.1]{M13}.
Applying the It\^{o} formula to $\phi_{m}(\Delta_{k}(t))$ gives
\begin{equation}
\phi_{m}(\Delta_{k}(t))=\sum\limits _{n=1}^{5}\int\limits _{0}^{t}\mathcal{R}_{k,m}^{n}(s)ds+\mathcal{M}_{k,m}(t),\label{EQ:03}
\end{equation}
where $\mathcal{R}_{k,m}^{1},\dots,\mathcal{R}_{k,m}^{5}$ are given by
\begin{align*}
\mathcal{R}_{k,m}^{1}(s) &= \phi_{m}'(\Delta_{k}(s))\left(\widetilde{b}_{kk}\Delta_{k}(s)-\sum\limits _{j\neq k}\widetilde{b}_{kj}X_{j}(s)\right)
\\ \mathcal{R}_{k,m}^{2}(s) &= c_{k}\phi_{m}''(\Delta_{k}(s))\left(\sqrt{Y_{k}}(s)-\sqrt{X_{k}(s)}\right)^{2}
\\ \mathcal{R}_{k,m}^{3}(s) &=\int\limits _{|z|\leq1}\int\limits _{\R_{+}}\left(\phi_{m}(\Delta_{k}(s)+z_{k}\delta_{k}(r,s))-\phi_{m}(\Delta_{k}(s))-z_{k}\delta_{k}(r,s)\phi_{m}'(\Delta_{k}(s))\right)dr\mu_{k}(dz)
\\ \mathcal{R}_{k,m}^{4}(s) &=\int\limits _{|z|>1}\int\limits _{\R_{+}}\left(\phi_{m}(\Delta_{k}(s)+z_{k}\delta_{k}(r,s))-\phi_{m}(\Delta_{k}(s))\right)dr\mu_{k}(dz)
\\ \mathcal{R}_{k,m}^{5}(s) &=\sum\limits _{j\neq k}\int\limits _{\R_{+}^{d}}\int\limits _{\R_{+}}\left(\phi_{m}(\Delta_{k}(s)-z_{k}\1_{\{r\leq X_{j}(s)\}})-\phi_{m}(\Delta_{k}(s))\right)dr\mu_{j}(dz),
\end{align*}
$(\mathcal{M}_{k,m}(t))_{t \geq 0}$ is a local martingale
and $\delta_{k}(r,s)=\1_{\{r\leq Y_{k}(s)\}}-\1_{\{r\leq X_{k}(s)\}}$.
For $l\in\N$, define the stopping time
\[
\tau_{l}=\inf\{t>0\ |\ \max\limits _{i\in\{1,\dots,d\}}\max\{X_{i}(t),Y_{i}(t)\}>l\}.
\]
Using the precise form of $\mathcal{M}_{k,m}$ given by It\^{o}'s formula combined with similar estimates to \cite[Lemma 4.1]{BLP15},
one can show that $(\mathcal{M}_{k,m}(t\wedge\tau_{l}))_{t\geq0}$
is a martingale for any $l\in\N$. Next we will prove that there exists
a constant $C>0$ such that
\begin{equation}
\sum_{n=1}^{5}\mathcal{R}_{k,m}^{n}(s) \leq C\Delta_{k}(s)_{+}+\frac{C}{m}.\label{EQ:01}
\end{equation}
Taking then expectations in \eqref{EQ:03}, using that $(\mathcal{M}_{k,m}(t\wedge\tau_{l}))_{t\geq0}$
is a martingale and estimating as in \eqref{EQ:01}, gives
\begin{align*}
\E[\phi_{m}(\Delta_{k}(t\wedge\tau_{l}))] & =\sum\limits _{n=1}^{5}\E\left[\int\limits _{0}^{t\wedge\tau_{l}}\mathcal{R}_{k,m}^{n}(s)ds\right]\leq C\E\left[\int\limits _{0}^{t\wedge\tau_{l}}\Delta_{k}(s)_{+}ds\right]+\frac{C}{m}\E[t\wedge\tau_{l}]\\
 & \leq C\int\limits _{0}^{t}\E[\Delta_{k}(s\wedge\tau_{l})_{+}]ds+\frac{Ct}{m}.
\end{align*}
Letting $m\to\infty$ and using property (i) gives
\[
\E[\Delta_{k}(t\wedge\tau_{l})_{+}]\leq C\int\limits _{0}^{t}\E[\Delta_{k}(s\wedge\tau_{l})_{+}]ds.
\]
Applying Gronwall lemma shows that, for any $k\in\{1,\dots,d\}$ and
$l\in\N$, one has $\E[\Delta_{k}(t\wedge\tau_{l})_{+}]=0$. Letting
now $l\to\infty$ proves the assertion.

Hence it remains to prove \eqref{EQ:01}. In order to estimate $\mathcal{R}_{k,m}^{1}$
we use properties (ii), (iii), $\widetilde{b}_{kj}\geq0$ for $k\neq j$
and $X_{j}(s) \geq 0$ to obtain
\[
\mathcal{R}_{k,m}^{1}(s)=\phi_{m}'(\Delta_{k}(s))\widetilde{b}_{kk}\Delta_{k}(s)_{+}-\phi_{m}'(\Delta_{k}(s))\sum\limits _{j\neq k}\widetilde{b}_{kj}X_{j}(s)\leq|\widetilde{b}_{kk}|\Delta_{k}(s)_{+}.
\]
For $\mathcal{R}_{k,m}^{2}$ we obtain from (iv) the estimate $\mathcal{R}_{k,m}^{2}(s)\leq\frac{2c_{k}}{m}$.
Let us now turn to $\mathcal{R}_{k,m}^{3}$. Using property (iv) we
see that, for each $y>0$, $z\geq0$ and $m\in\N$, 
there exists $\vartheta=\vartheta(y,z)\in[0,1]$
such that
\[
\phi_{m}(y+z)-\phi_{m}(y)-\phi_{m}'(y)z=\phi_{m}''(y+\vartheta z)\frac{z^{2}}{2}\leq\frac{2z^{2}}{2m(y+\vartheta z)}\leq\frac{z^{2}}{my}.
\]
Next observe that $\delta_{k}(r,s)>0$ if and only if $\Delta_{k}(s)>0$
and $r\in(X_{k}(s),Y_{k}(s)]$. 
Applying both observations, we obtain
\begin{align*}
\mathcal{R}_{k,m}^{3}(s) & \leq\1_{\{\Delta_{k}(s)>0\}}\int\limits _{|z|\leq1}\int\limits _{\R_{+}}\left(\phi_{m}(\Delta_{k}(s)+z_{k}\delta_{k}(r,s))-\phi_{m}(\Delta_{k}(s))-z_{k}\delta_{k}(r,s)\phi_{m}'(\Delta_{k}(s))\right)dr\mu_{k}(dz)\\
 & \leq\frac{\1_{\{\Delta_{k}(s)>0\}}}{m\Delta_{k}(s)}\int\limits _{|z|\leq1}\int\limits _{\R_{+}}z_{k}^{2}\delta_{k}(r,s)^{2}dr\mu_{k}(dz)
 \leq\frac{1}{m}\int\limits _{|z|\leq1}z_{k}^{2}\mu_{k}(dz),
\end{align*}
where we have used $\int_{\R_{+}}\delta_{k}(r,s)^{2}dr=\Delta_{k}(s)$
a.s. on $\{\Delta_{k}(s)>0\}$. For $\mathcal{R}_{k,m}^{4}$ we use
property (ii), so that
\begin{align*}
\mathcal{R}_{k,m}^{4}(s) & \leq\1_{\{\Delta_{k}(s)>0\}}\int\limits _{|z|>1}\int\limits _{\R_{+}}\left(\phi_{m}(\Delta_{k}(s)+z_{k}\delta_{k}(r,s))-\phi_{m}(\Delta_{k}(s))\right)\mu_{k}(dz)dr\\
 & \leq\1_{\{\Delta_{k}(s)>0\}}\int\limits _{|z|>1}\int\limits _{\R_{+}}z_{k}\delta_{k}(r,s)\mu_{k}(dz)dr
 \leq\Delta_{k}(s)_{+}\int\limits _{|z|>1}z_{k}\mu_{k}(dz),
\end{align*}
where we have also used $\int_{\R_{+}}\delta_{k}(r,s)dr=\Delta_{k}(s)$.
For the last term we use property (ii), so that $\mathcal{R}_{k,m}^{5}(s)\leq0$.
This proves \eqref{EQ:01} and hence the assertion. \end{proof}

\subsection{Proofs of Theorem \ref{MAIN:THEOREM} and Theorem \ref{MAIN:THEOREM1}}

We are now prepared to prove our main results of this work.
First observe that Proposition \ref{COMPARISON} implies that, for
any $k\in\{1,\dots,d\}$,
\[
\P[ Y_k(t) > 0, \quad t \geq 0] = 1\ 
\Rightarrow\ \P[ X_k(t) > 0, \quad t \geq 0]=1,
\]
and similarly
\[
\P[\lim\limits _{t\to\infty}Y_{k}(t)=\infty]=1\ \Rightarrow\ \P[\lim\limits _{t\to\infty}X_{k}(t)=\infty]=1,
\]
where $X$ and $Y$ are the unique solutions to \eqref{SDE:CBI} and
\eqref{EQ:12}, respectively. In view of Proposition \ref{LEMMA:00}, $Y_{k}$
satisfies the conditions of \cite[Corollary 6]{FU13} or 
\cite[Theorem 2]{DFM14}, respectively. 
Now it is easy to see that the assertions of Theorem \ref{MAIN:THEOREM} and Theorem \ref{MAIN:THEOREM1} are true.

\section*{Appendix: Additional proofs}

\begin{proof}[Proof of Proposition \ref{PROP:00}.]
 Observe that under condition
\eqref{EQ:07} the process $X_{k}$ also satisfies
\begin{align*}
X_{k}(t) & =x_{k}+\int\limits _{0}^{t}\left(\beta_{k}+\sum\limits _{j=1}^{d}g_{kj}X_{j}(s)\right)ds+\int\limits _{0}^{t}\int\limits _{\R_{+}^{d}}z_{k}N_{\nu}(ds,dz)\\
 & \ \ \ +\int\limits _{0}^{t}\int\limits _{\R_{+}^{d}}\int\limits _{\R_{+}}z_{k}\1_{\{r\leq X_{k}(s-)\}}N_{k}(ds,dz,dr)+\sum\limits _{j\neq k}\int\limits _{0}^{t}\int\limits _{\R_{+}^{d}}\int\limits _{\R_{+}}z_{k}\1_{\{r\leq X_{j}(s-)\}}N_{j}(ds,dz,dr),
\end{align*}
where $g_{kj}$ is defined in \eqref{GDEF}.
This implies that $X_{k}$
has bounded variation. Let $y(t)$ be the unique solution to
$y(t)=x_{k}+\int_{0}^{t}\left(\beta_{k}+\theta_{k}y(s)\right)ds$, i.e.,
\[
y(t)=\begin{cases}
x_{k}e^{\theta_{k}t}+\beta_{k}\frac{e^{\theta_{k}t}-1}{\theta_{k}}, & \text{ if }\theta_{k}\neq0\\
x_k + \beta_{k}t, & \text{ if }\theta_{k}=0
\end{cases},\qquad t\geq0.
\]
Proceeding exactly as in the proof of Proposition \ref{COMPARISON},
we obtain $\P[X_{k}(t)\geq y(t)]=1$ for all $t\geq0$. This proves
the assertion. 
\end{proof}

\begin{proof}[Proof of Proposition \ref{PROP:01}.]
Observe that under \eqref{EQ:07}
the process $X$ also satisfies
\begin{equation*}
X(t) = x+\int\limits _{0}^{t}\left(\beta+GX(s)\right)ds+\int\limits _{0}^{t}\int\limits _{\R_{+}^{d}}zN_{\nu}(ds,dz)+\sum\limits _{j=1}^{d}\int\limits _{0}^{t}\int\limits _{\R_{+}^{d}}\int\limits _{\R_{+}}z\1_{\{r\leq X_{j}(s-)\}}N_{j}(ds,dz,dr).
\end{equation*}
Let $y(t)$ be the unique solution to
$y(t)=x+\int_{0}^{t}\left(\beta+Gy(s)\right)ds$ which is given by
$y(t)=e^{tG}x+\int_{0}^{t}e^{sG}\beta ds$.
Proceeding exactly as in the proof of Proposition \ref{COMPARISON},
we obtain $\P[X_{k}(t)\geq y_{k}(t)]=1$ for all $t\geq0$ and $k\in\{1,\dots,d\}$.
This proves the assertion.
\end{proof}

\begin{proof}[Proof of Remark \ref{CORR:01}]
 Set $\kappa = \max\{M_0, M_1, M_2\}$.
 If $\alpha_{k}<1+\gamma_{k}$, then $\frac{F^{(k)}(u)}{R^{(k)}(u)} \geq \frac{C_1}{C_2}u^{\gamma_k - \alpha_k}$, for $u\in[\kappa,\xi]$, and hence
\begin{align*}
\exp\left(\int\limits _{\kappa}^{\xi}\frac{F^{(k)}(u)}{R^{(k)}(u)}du\right) 
 &\geq\exp\left(\frac{C_{1}}{C_{2}}\int\limits _{\kappa}^{\xi}u^{\gamma_{k}-\alpha_{k}}du\right)
 \\ &=\exp\left(-\frac{C_{1}}{C_{2}}\frac{\kappa^{1+\gamma_{k}-\alpha_{k}}}{1+\gamma_{k}-\alpha_{k}}\right)\exp\left(\frac{C_{1}}{C_{2}}\frac{\xi^{1+\gamma_{k}-\alpha_{k}}}{1+\gamma_{k}-\alpha_{k}}\right)
\end{align*}
and 
\begin{align*}
\int\limits _{\kappa}^{\infty}\exp\left(\int\limits _{\kappa}^{\xi}\frac{F^{(k)}(u)}{R^{(k)}(u)}du\right)\frac{d\xi}{R^{(k)}(\xi)} 
&\geq\frac{\exp\left(-\frac{C_{1}}{C_{2}}\frac{\kappa^{1+\gamma_{k}-\alpha_{k}}}{1+\gamma_{k}-\alpha_{k}}\right)}{C_{2}}\int\limits _{\kappa}^{\infty}\exp\left(\frac{C_{1}}{C_{2}}\frac{\xi^{1+\gamma_{k}-\alpha_{k}}}{1+\gamma_{k}-\alpha_{k}}\right)\frac{d\xi}{\xi^{\alpha_{k}}}=\infty.
\end{align*}
This proves \eqref{EQ:04} under (a).
If $\alpha_{k}=1+\gamma_{k}$, then we obtain for $\xi\geq \kappa$
and $u\in[\kappa,\xi]$,
\begin{align*}
\exp\left(\int\limits _{\kappa}^{\xi}\frac{F^{(k)}(u)}{R^{(k)}(u)}du\right) 
 \geq \exp\left(\frac{C_{1}}{C_{2}}\int\limits _{\kappa}^{\xi}u^{\gamma_{k}-\alpha_{k}}du\right)
 = \kappa^{-\frac{C_{1}}{C_{2}}}\xi^{\frac{C_{1}}{C_{2}}}.
\end{align*}
Using $\alpha_{k}\leq 1+\frac{C_{1}}{C_{2}}$ gives
\begin{equation*}
\int\limits _{\kappa}^{\infty}\exp\left(\int\limits _{\kappa}^{\xi}\frac{F^{(k)}(u)}{R^{(k)}(u)}du\right)\frac{d\xi}{R^{(k)}(\xi)} 
 \geq\frac{\kappa^{-\frac{C_{1}}{C_{2}}}}{C_{2}}\int\limits_{\kappa}^{\xi}\frac{\xi^{\frac{C_{1}}{C_{2}}}}{\xi^{\alpha_{k}}}d\xi=\infty,
\end{equation*}
and hence proves \eqref{EQ:04} under (b). 
\end{proof}

\begin{footnotesize}

\bibliographystyle{alpha}
\bibliography{Bibliography}

\begin{thebibliography}{CPGUB13}

\bibitem[Alf15]{A15}
Aur\'{e}lien Alfonsi.
\newblock {\em Affine diffusions and related processes: simulation, theory and
  applications}, volume~6 of {\em Bocconi \& Springer Series}.
\newblock Springer, Cham; Bocconi University Press, Milan, 2015.

\bibitem[BLP15]{BLP15}
M\'aty\'as Barczy, Zenghu Li, and Gyula Pap.
\newblock Stochastic differential equation with jumps for multi-type continuous
  state and continuous time branching processes with immigration.
\newblock {\em ALEA Lat. Am. J. Probab. Math. Stat.}, 12(1):129--169, 2015.

\bibitem[BPP18a]{BPP18b}
M\'{a}ty\'{a}s Barczy, Sandra Palau, and Gyula Pap.
\newblock Almost sure, {$L_1$}- and {$L_2$}-growth behavior of supercritical
  multi-type continuous state and continuous time branching processes with
  immigration.
\newblock {\em arXiv:1803.10176 [math.PR]}, 2018.

\bibitem[BPP18b]{BPP18}
M\'{a}ty\'{a}s Barczy, Sandra Palau, and Gyula Pap.
\newblock Asymptotic behavior of projections of supercritical multi-type
  continuous state and continuous time branching processes with immigration.
\newblock {\em arXiv:1806.10559 [math.PR]}, 2018.

\bibitem[CLP18]{CLP18}
Marie Chazal, Ronnie Loeffen, and Pierre Patie.
\newblock Smoothness of continuous state branching with immigration semigroups.
\newblock {\em J. Math. Anal. Appl.}, 459(2):619--660, 2018.

\bibitem[CPGUB13]{CPU13}
M.~Emilia Caballero, Jos\'{e}~Luis P\'{e}rez~Garmendia, and Ger\'{o}nimo
  Uribe~Bravo.
\newblock A {L}amperti-type representation of continuous-state branching
  processes with immigration.
\newblock {\em Ann. Probab.}, 41(3A):1585--1627, 2013.

\bibitem[DFM14]{DFM14}
Xan Duhalde, Cl\'{e}ment Foucart, and Chunhua Ma.
\newblock On the hitting times of continuous-state branching processes with
  immigration.
\newblock {\em Stochastic Process. Appl.}, 124(12):4182--4201, 2014.

\bibitem[DFS03]{DFS03}
Darrell Duffie, Damir Filipovi\'c, and Walter Schachermayer.
\newblock Affine processes and applications in finance.
\newblock {\em Ann. Appl. Probab.}, 13(3):984--1053, 2003.

\bibitem[Fel51]{F51}
William Feller.
\newblock Diffusion processes in genetics.
\newblock In {\em Proceedings of the {S}econd {B}erkeley {S}ymposium on
  {M}athematical {S}tatistics and {P}robability, 1950}, pages 227--246.
  University of California Press, Berkeley and Los Angeles, 1951.

\bibitem[FFS85]{FFS85}
P.~J. Fitzsimmons, Bert Fristedt, and L.~A. Shepp.
\newblock The set of real numbers left uncovered by random covering intervals.
\newblock {\em Z. Wahrsch. Verw. Gebiete}, 70(2):175--189, 1985.

\bibitem[FJR18a]{FJR18a}
Martin Friesen, Peng Jin, and Barbara R\"udiger.
\newblock Existence of densities for multi-type {CBI} processes.
\newblock {\em arXiv:1810.00400 [math.PR]}, 2018.

\bibitem[FJR18b]{FJR18b}
Martin Friesen, Peng Jin, and Barbara R\"udiger.
\newblock Existence of densities for stochastic differential equations driven
  by a {L}\'evy process with anisotropic jumps.
\newblock {\em arXiv:1810.07504 [math.PR]}, 2018.

\bibitem[FJR18c]{FJR18c}
Martin Friesen, Peng Jin, and Barbara R\"udiger.
\newblock Stochastic equation and exponential ergodicity in {W}asserstein
  distances for affine processes.
\newblock 2018.

\bibitem[FL10]{MR2584896}
Zongfei Fu and Zenghu Li.
\newblock Stochastic equations of non-negative processes with jumps.
\newblock {\em Stochastic Process. Appl.}, 120(3):306--330, 2010.

\bibitem[FMS13]{FMS13}
Damir Filipovi\'{c}, Eberhard Mayerhofer, and Paul Schneider.
\newblock Density approximations for multivariate affine jump-diffusion
  processes.
\newblock {\em J. Econometrics}, 176(2):93--111, 2013.

\bibitem[FUB14a]{FU13}
Cl\'{e}ment Foucart and Ger\'{o}nimo Uribe~Bravo.
\newblock Local extinction in continuous-state branching processes with
  immigration.
\newblock {\em Bernoulli}, 20(4):1819--1844, 2014.

\bibitem[FUB14b]{FU14}
Cl\'{e}ment Foucart and Ger\'{o}nimo Uribe~Bravo.
\newblock Local extinction in continuous-state branching processes with
  immigration.
\newblock {\em Bernoulli}, 20(4):1819--1844, 2014.

\bibitem[Gre74]{G74}
D.~R. Grey.
\newblock Asymptotic behaviour of continuous time, continuous state-space
  branching processes.
\newblock {\em J. Appl. Probability}, 11:669--677, 1974.

\bibitem[GZ18]{GZ18}
Peter~W. Glynn and Xiaowei Zhang.
\newblock Affine jump-diffusions: Stochastic stability and limit theorems.
\newblock {\em arXiv:1811.00122 [q-fin.MF]}, 2018.

\bibitem[JKR18]{JKR18}
Peng Jin, Jonas Kremer, and Barbara R\"udiger.
\newblock Existence of limiting distribution for affine processes.
\newblock {\em arXiv:1812.05402 [math.PR]}, 2018.

\bibitem[KS17]{KS17}
Kamil Kaleta and Pawe\l\; Sztonyk.
\newblock Small-time sharp bounds for kernels of convolution semigroups.
\newblock {\em J. Anal. Math.}, 132:355--394, 2017.

\bibitem[Kyp06]{K06}
Andreas~E. Kyprianou.
\newblock {\em Introductory lectures on fluctuations of {L}\'{e}vy processes
  with applications}.
\newblock Universitext. Springer-Verlag, Berlin, 2006.

\bibitem[Lam67]{L67b}
John Lamperti.
\newblock Continuous state branching processes.
\newblock {\em Bull. Amer. Math. Soc.}, 73:382--386, 1967.

\bibitem[Li11]{L11}
Zenghu Li.
\newblock {\em Measure-valued branching {M}arkov processes}.
\newblock Probability and its Applications (New York). Springer, Heidelberg,
  2011.

\bibitem[Ma13]{M13}
Ru~Gang Ma.
\newblock Stochastic equations for two-type continuous-state branching
  processes with immigration.
\newblock {\em Acta Math. Sin. (Engl. Ser.)}, 29(2):287--294, 2013.

\bibitem[MSV18]{MSV18}
Eberhard Mayerhofer, Robert Stelzer, and Johanna Vestweber.
\newblock Geometric ergodicity of affine processes on cones.
\newblock {\em arXiv:1811.10542 [math.PR]}, 2018.

\bibitem[Par16]{P16}
\'{E}tienne Pardoux.
\newblock {\em Probabilistic models of population evolution}, volume~1 of {\em
  Mathematical Biosciences Institute Lecture Series. Stochastics in Biological
  Systems}.
\newblock Springer, [Cham]; MBI Mathematical Biosciences Institute, Ohio State
  University, Columbus, OH, 2016.
\newblock Scaling limits, genealogies and interactions.

\bibitem[SW73]{SW73}
Tokuzo Shiga and Shinzo Watanabe.
\newblock Bessel diffusions as a one-parameter family of diffusion processes.
\newblock {\em Z. Wahrscheinlichkeitstheorie und Verw. Gebiete}, 27:37--46,
  1973.

\end{thebibliography}

\end{footnotesize}

\end{document}